\documentclass[11pt]{article}

\usepackage{amsmath,amsthm,amscd}

\usepackage{amssymb,latexsym}

\usepackage{flexisym}

\usepackage{enumerate}

\usepackage{supertabular}

\usepackage{hyperref}

\usepackage{refcount}

\usepackage[titletoc,title]{appendix}

\newtheorem{proposition}{Proposition}[section]
\newtheorem{theorem}[proposition]{Theorem}

\numberwithin{equation}{section}

\def\Q{{\mathbb Q}}
\def\Z{{\mathbb Z}}

\def\gcd{\mathrm{gcd}}

\begin{document}

\baselineskip=17pt

\title{A note on the Diophantine equations $x^{2}\pm5^{\alpha}\cdot p^{n}=y^{n}$}
\author{
G\"{O}KHAN~SOYDAN
}


\date{} 

\maketitle

\begin{abstract}
Suppose that $x$ is odd, $n\geq7$ and $p\notin\{2,5\}$ are primes. In this paper, we prove that the Diophantine equations $x^{2}\pm5^{\alpha}p^{n}=y^{n}$ have no solutions in positive integers $\alpha,x,y$ with $gcd(x,y)=1$.
\end{abstract}

\renewcommand{\thefootnote}{}



\footnote{This work was supported by the Research Fund of Uluda\u{g} University under project numbers: 2015/23, 2016/9.}

\renewcommand{\thefootnote}{\arabic{footnote}}

\setcounter{footnote}{0}

\section{Introduction}
            \label{sec intro}
The Diophantine equation
\begin{equation}\label{eq.1}
x^{2}+B=y^{n}, \quad x,y\geq1, \quad n\geq 3,
\end{equation}
where $B$ is a product of at least two prime powers were studied in some recent papers. First we assume that $q$ is an odd prime. All solutions of the Diophantine equation \eqref{eq.1} where $B=2^{a}q^{b}$ were given in \cite{Lu} for $q=3$, in \cite{LT1} for $q=5$, in \cite{CDLPS} for $q=11$, in \cite{LT2} for $q=13$, in \cite{Dab} for $q=17,29,41$ , in \cite{SUZ} for $q=19$. Next assume that $q$ is a general odd prime. In \cite{ZLST}, Zhu, Le, Soydan and T\'{o}gbe gave all the solutions of the equation 
$
x^{2}+2^{a}q^{b}=y^{n}, \; x\geq1, y>1,\; \gcd(x,y)=1,\; a\geq0, b>0, \; n\geq3
$
under some conditions.

Many authors also considered the Diophantine equation \eqref{eq.1} where $B$ is a product of at least two distinct odd primes. The cases $B=5^{a}13^{b}$ and $B=5^{a}17^{b}$ when $x$ and $y$ are coprime were solved completely in \cite{LMT} and \cite{PR}, respectively. In 2010, the complete solution $(n,a,b,x,y)$ of the Diophantine equation \eqref{eq.1} for the case $B=5^{a}11^{b}$ when $gcd(x,y)=1,$ except for the case when $abx$ is odd, was given by Cangul, Demirci, Soydan and Tzanakis, \cite{CDST}. Six years later, the remaining case of the Diophantine equation \eqref{eq.1} for the case $B=5^{a}11^{b}$ were covered by Soydan and Tzanakis, \cite{ST}. All solutions of the Diophantine equation \eqref{eq.1} for the cases $B=7^{a}11^{b}$ -except for the case when $ax$ is odd and $b$ is even-, $B=11^{a}17^{b}$, $B=2^{a}5^{b}13^{c}$, $B=2^{a}3^{b}11^{c}$, $B=2^{a}5^{b}17^{c}$ and $B=2^{a}3^{b}17^{c}$ - $2^{a}13^{b}17^{c}$ can be found in \cite{S1}-\cite{S2}, \cite{BP}, \cite{GLT}, \cite{CDILS}, \cite{GMT1} and \cite{GMT2}, respectively. In \cite{P}, Pink gave all the non-exceptional solutions of the equation \eqref{eq.1} (according to terminology of that paper) for the case $B=2^{a}3^{b}5^{c}7^{d}$. For a survey concerning equation \eqref{eq.1} see \cite{BP}, \cite{AB}.

Now we assume that $n\geq7$ and $p\notin\{2,5\}$ are primes. Here we consider the Diophantine equations
\begin{equation} \label{eq.2}
x^2+5^\alpha p^n=y^n  
\end{equation}
and
\begin{equation} \label{eq.3}
x^2-5^\alpha p^n=y^n
\end{equation}
where $x,y\geq1$, $\alpha\geq 0$ and $gcd(x,y)=1.$ There are many papers concerning partials solutions for the equations \eqref{eq.2} and \eqref{eq.3}. The known results except the ones mentioned above include the following theorem.
\begin{theorem}\label{theo.1}

$(i)$ Let $p>7$ be an odd prime with $p\not\equiv 7\ (\textrm{mod}\ 8)$ and $(n,h_{0})=1$ where $h_{0}$ denote the class number of the field $\Q(\sqrt{-p})$. Under these conditions if $\alpha=0$, then the equation \eqref{eq.2} has no solutions.
\\$(ii)$ Let $p>2$ be a prime. If $\alpha=0$, then \eqref{eq.3} has no solutions.
\end{theorem}
\begin{proof}
$(i)$ See \cite{AA2}.

$(ii)$ See \cite{DM}.
\end{proof}
Our main result is following.
\begin{theorem}\label{title-eq}
Suppose that $x$ is odd, $n\geq7$ and $p\notin\{2,5\}$ are primes. Then the Diophantine equations
\begin{equation} \label{eq.4}
x^2+5^\alpha p^n=y^n  
\end{equation}
and
\begin{equation} \label{eq.5}
x^2-5^\alpha p^n=y^n
\end{equation}
have no solutions in positive integers $\alpha, x , y$ with $gcd(x,y)=1$.  	
\end{theorem}
Here the equation \eqref{eq.4} is an extension of the equation \eqref{eq.1} the cases when $B=5^{a}11^{b}$, $B=5^{a}13^{b}$, $B=5^a17^b$ in \cite{CDST}, \cite{ST}, \cite{LMT}, \cite{PR}, respectively.
\section{Preliminaries}
This section introduces some well known notions and results that will be used to prove the main result.
\subsection{The modular method}
The most important progress in the field of the Diophantine equations has been with Wile's proof of Fermat's Last Theorem \cite{Wiles}. His proof is based on deep results about Galois representations associated to elliptic curves and modular forms. The method of using such results to deal with Diophatine problems, is called the \textit{modular method}. Especially modular method is useful to solve Diophantine equations of the form
\begin{equation*}
ax^p+by^p=cz^p,\quad ax^p+by^p=cz^2,\quad ax^p+by^p=cz^3, ...   (p\, \textrm{prime}).
\end{equation*}
Modular method follows these steps: associate to a (hypotetical) solution of such a Diophantine equation a certain elliptic curve, called a \textit{Frey curve}, with discriminant an explicitly known constant times a $p$-th power. Next (under some technical assumptions) apply Ribbet's level lowering theorem \cite{Ri} to show that Galois representation on the $p$-torsion of the Frey curve occurs from a newform of weight 2 and a fairly small level $N$ say. If there are no such newforms then there are no non-trivial\footnote{A solution to the equation $ax^p+by^p=cz^r$ with $a,b,c \in \Z/\{0\}$, $x,y,z \in \Z$, $p,q,r \in \Z_{\geqslant2}$ is called nontrivial if $xyz\neq0$.} solutions to the original Diophantine equation. 

Now we stop here, since we only need some of these steps of the modular method in this work (For the details concerning modular method see \cite[Chapter 15]{Cohen} and \cite{Si}). 
\subsection{Signature $(n,n,2)$}
Here we follow the paper of Siksek  \cite[Section 14]{Si} and we give recipes for signature $(n,n,2)$ which was firstly described by Bennett and Skinner \cite{BS}. (See also \cite{IK}).

Assume that $n\geq7$ is prime and $a, b, c, A, B$ and $C$ are nonzero integers with $Aa$, $Bb$ and $Cc$ pairwise coprime, satisfying
\begin{equation} \label{eq.2.1}
Aa^n+Bb^n=Cc^2.  
\end{equation}
We suppose that
\begin{equation}\label{eq.2.2}
ord_{r}(A)<n, \quad ord_{r}(B)<n  \quad \textrm{for all primes}  \, r 
\end{equation}
and
\begin{equation*}
C \quad \textrm{is squarefree}.
\end{equation*}

With assumptions and notation as above without loss of generality, we may suppose we are in one of the following situations:

$(i)$ $abABC\equiv 1\ (\textrm{mod}\ 2)$ and $b\equiv -BC\ (\textrm{mod}\ 4)$.

$(ii)$ $ab\equiv 1\ (\textrm{mod}\ 2)$ and either $ord_{2}(B)=1$ or $ord_{2}(C)=1$.

$(iii)$  $ab\equiv 1\ (\textrm{mod}\ 2)$, $ord_{2}(B)=2$ and $C\equiv -bB/4\ (\textrm{mod}\ 4)$.

$(iv)$ $ab\equiv 1\ (\textrm{mod}\ 2)$, $ord_{2}(B)\in\{3,4,5\}$ and $c\equiv C\ (\textrm{mod}\ 4)$.

$(v)$ $ord_{2}(bB^n)\geq6$ and $c\equiv C\ (\textrm{mod}\ 4)$.

In cases $(i)$ and $(ii)$, we will consider the curve
\begin{equation}\label{eq.2.3}
E_{1}(a,b,c): Y^2=X^3+2cCX^2+BCb^nX.
\end{equation}
In cases $(iii)$ and $(iv)$, we will consider
\begin{equation}\label{eq.2.4}
E_{2}(a,b,c): Y^2=X^3+cCX^2+\dfrac{BCb^n}{4}X,
\end{equation}
in case $(v)$,
\begin{equation}\label{eq.2.5}
E_{3}(a,b,c): Y^2+XY=X^3+\dfrac{cC-1}{4}X^2+\dfrac{BCb^n}{64}X.
\end{equation}
These are all elliptic curves defined over $\Q.$

The following theorem \cite[Theo. 16]{Si} summarizes some useful fact about these curves.
\begin{theorem} (Bennett and Skinner, \cite{BS}) \label{theo.2.1}
Let $i=1,2$ or $3$.

$(a)$ The discriminant $\Delta(E)$ of the curve $E=E_{i}(a,b,c)$ is given by
\begin{equation*}
\Delta(E)=2^{\delta_{i}}C^3B^2A(ab^2)^n
\end{equation*}
where
\[
\delta_{i}=\begin{cases}
6 & \mbox{if $i=1$} \\
0 & \mbox{if $i=2$}\\
-12 & \mbox{if $i=3.$}
\end{cases}
\]

$(b)$ The conductor $N(E)$ of the curve $E=E_{i}(a,b,c)$ is given by
\begin{equation*}
N(E)=2^\alpha C^2\prod_{s|abAB} s  \quad (s \, \, \textrm{is odd prime})
\end{equation*}
where
\[
\alpha=\begin{cases}
5 & \mbox{if $i=1$}, \, \mbox{case $(i)$} \\
6 & \mbox{if $i=1$}, \, \mbox{case $(ii)$} \\
1 & \mbox{if $i=2$}, \, \mbox{case $(iii)$}, \, \, ord_{2}(B)=2 \, \, \mbox{and}  \, \, \, b\equiv -BC/4\ (\textrm{mod}\ 4)  \\
2 & \mbox{if $i=2$}, \, \mbox{case $(iii)$}, \, \, ord_{2}(B)=2 \, \, \mbox{and}  \, \, \, b\equiv BC/4\ (\textrm{mod}\ 4)  \\
4 & \mbox{if $i=2$}, \, \mbox{case $(iv)$} \, \, \mbox{and} \, \, ord_{2}(B)=3  \\
2 & \mbox{if $i=2$}, \, \mbox{case $(iv)$}\, \, \mbox{and} \, \, ord_{2}(B)\in \{4,5\} \\
-1 & \mbox{if $i=3$}, \, \mbox{case $(v)$} \, \, \mbox{and} \, \, ord_{2}(Bb^n)=6  \\
0 & \mbox{if $i=3$}, \, \mbox{case $(v)$} \, \, \mbox{and} \, \, ord_{2}(Bb^n)\geq7.  \\
\end{cases}
\]

$(c)$ Suppose that $E=E_{i}(a,b,c)$ does not have complex multiplication (This would follow if we assume that $ab\neq\pm1$). Then $E=E_{i}(a,b,c)\thicksim_n f$ for some newform
$f$ of level
\begin{equation*}
N_{n}=2^\beta C^2\prod_{t|AB} t  \quad (t \, \, \textrm{is odd prime})
\end{equation*}
where
\[
\beta=\begin{cases}
\alpha & \mbox{cases $(i)$-$(iv)$},\\
0 & \mbox{case $(v)$} \, \, \mbox{and} \, \, ord_{2}(B)\neq0,6, \\
1 & \mbox{case $(v)$} \, \, \mbox{and}\, \, ord_{2}(B)=0, \\
-1 &  \mbox{case $(v)$} \, \, \mbox{and}\, \, ord_{2}(B)=6. \\
\end{cases}
\]

$(d)$ The curves $E_{i}(a,b,c)$ have non-trivial 2-torsion.  
\end{theorem}
Finally we give an important result \cite[Theo. 1]{Si} about newforms.
\begin{theorem}\label{theo.2.2}
There are no newforms at levels 1, 2, 3, 4, 5, 7, 8, 9, 10, 12, 13, 16, 18, 22, 25, 28, 60.
\end{theorem}
Now we are ready to prove Theorem \ref{title-eq}.
\section{The proof of Theorem  \ref{title-eq}}
First suppose that $(x,y,\alpha,p,n)$ is a solution to \eqref{eq.4} where $x$ is odd, $n\geq7$ and $p\notin\{2,5\}$ are primes. Thus the equation \eqref{eq.4} becomes
\begin{equation} \label{eq.3.1}
(-5)^\alpha p^n+y^n=x^2
\end{equation} 
with $2\nmid\alpha$. We may assume without loss of generality that $x\equiv 1\ (\textrm{mod}\ 4)$. With the notation in \eqref{eq.2.1}, we see that \eqref{eq.3.1} is a ternary equation of signature $(n,n,2)$. We have the following notations which satisfy \eqref{eq.2.2}
\begin{equation*}
A=(-5)^\alpha, \, \, B=1, \, \, C=1,\, \, a=p, \, \, b=y,\, \, c=x. 
\end{equation*}
Since $y$ is even, $x\equiv 1\ (\textrm{mod}\ 4)$ and $n\geq7$, then with the case $(v)$ (in page $4$) we are interested in the following elliptic curve (called a Frey curve)
\begin{equation*}
E_{3}: Y^2+XY=X^3+\dfrac{x-1}{4}X^2+\dfrac{y^n}{64}X.
\end{equation*}
According to the cases $(a)$ and $(b)$ of Theorem \ref{theo.2.1}, we write the discriminant and conductor of this elliptic curve, respectively
\begin{equation*}
\Delta(E_{3})=2^{-12}(-5)(ab^2)^n,\, \, N(E_{3})=\prod_{s|abAB} s=5 \prod_{s|ab} s 
\end{equation*}
where in the last product $s$ is odd prime. With the case $(c)$ of Theorem \ref{theo.2.1} we compute the level $N_{n}=2\prod_{t|AB} t=10$   ($\,t$ prime). But Theorem \ref{theo.2.2} tells us that there is no newform of level $10$. Thus we deduce the equation \eqref{eq.3.1} has no solutions where $x$ is odd, $p\notin\{2,5\}$ and $n\geq 7$ are primes.

For the case $2\mid\alpha$, we can write the equation \eqref{eq.5} as follows.
\begin{equation}\label{eq.3.2}
(-5)^\alpha p^n+y^n=x^2.
\end{equation}
Following same steps as the case $2\nmid\alpha$, we see that \eqref{eq.3.2} has no solutions where $p\notin\{2,5\}$ and $n\geq7$ are primes. So the proof of theorem is completed.

\bigskip 
\begin{tabular}{l}
Department of Mathematics,  \\Uluda\u{g} University,  \\16059 Bursa-TURKEY \\ \\E-mail address: gsoydan@uludag.edu.tr
\end{tabular}       

\begin{thebibliography}{99}
\normalsize
\baselineskip=17pt
\bibitem{AA2} S.~A.~Arif, F.S.~Abu~Muriefah {\em On the Diophantine equation $x^{2}+q^{2k+1}=y^{n}$}, J. Num. Th. {\bf 95} (2002), no.~1, 95-100.
%
\bibitem{AB} F.S.~Abu~Muriefah, Y.~Bugeaud,
{\em The Diophantine equation $x^{2}+C=y^{n}$: a brief overview},
Revis.~Col.~Math.  {\bf 40} (2006),  no.~1, 31-37.
%
\bibitem{BS} M.A.~Bennett, C.M.~Skinner, {\em Ternary Diophantine
equations via Galois representations and modular forms}, Canad.~J.~Math. {\bf 56} (2004), no.~1, 23-54.
%
\bibitem{BP} A.~Bercz\'{e}s, I.~Pink, {\em On generalized Lebesgue-Ramanujan-Nagell equations},
An. \c{S}t.~Univ.~Ovid.~Cons. {\bf 22} (2014),  no.~1, 51-71.
%
\bibitem{CDILS} I.N.~Cangul, M.~Demirci, F.~Luca, I.~Inam, G.~Soydan, {\em On the Diophantine equation 
$x^{2}+2^{a}3^{b}11^{c}=y^{n}$}, Math.~Slovaca {\bf 63} (2013), no.~3, 647-659.
%
\bibitem{CDLPS} I.N.~Cangul, M.~Demirci, F.~Luca, \'{A}.~Pint\'{e}r, G.~Soydan, {\em On the Diophantine equation $x^{2}+2^{a}11^{b}=y^{n}$}, Fibonacci Quart. {\bf 48} (2010), no.~1, 39-46.
%
\bibitem{CDST} I.N.~Cangul, M.~Demirci, G.~Soydan, N.~Tzanakis, 
{\em On the Diophantine equation $x^{2}+5^{a}11^{b}=y^{n}$},
Funct.~Approx. {\bf 43} (2010), no.~2, 209-225.
%
\bibitem{Cohen} H. Cohen, {Number Theory Vol. II: Analytic and Modern Tools, Springer, (2007).
%
\bibitem{Dab} A.~Dabrowski, {\em On the Lebesgue-Nagell equation}, Colloq.~Math. {\bf 125} (2011), no.~2, 245-253.
%
\bibitem{DM} H.~Darmon, L.~Merel, {\em Winding quotients and some variants of Fermat's Last Theorem}, Jour.~f\"{u}r~die~reine~und~ang.~Math. {\bf 490} (1997), 81-100.
%
\bibitem{GMT1} H.~Godinho, D.~Marques, A.~Togb\'{e}, {\em On the Diophantine equation $x^{2}+2^{\alpha}5^{\beta}17^{\gamma}=y^{n}$}, Com.~in~Math. {\bf 20} (2012), no.~2, 81-88.
%
\bibitem{GMT2} H.~Godinho, D.~Marques, A.~Togb\'{e}, {\em On the Diophantine equation $x^{2}+C=y^{n}$, $C=2.3.17$, $C=2.13.17$}, Math. Slovaca {\bf 66} (2016), no.~3, 1-10.
%
\bibitem{GLT} E.~Goins, F.~Luca, A.~Togb\'{e}, {\em On the Diophantine equation $x^{2}+2^{\alpha}5^{\beta}13^{\gamma}=y^{n}$}, ANTS~VIII~Proc. {\bf 5011} (2008), 430-442.
%
\bibitem{IK} W.~Ivorra, A.~Kraus, {\em Quelques r\'{e}sultats sur les \'{e}quations $ax^{p}+by^{p}=cz^{2}$}, Canad. J. Math. {\bf 58} (2006), no.~1, 115-153.
%
\bibitem{Lu} F.~Luca, {\em On the Diophantine equation $x^{2}+2^{a}3^{b}=y^{n}$}, Int.~J.~Math.~Sci. {\bf 29} (2002), no.~4, 239-244.
%
\bibitem{GL} Y.~Guo, M.~H.~Le, {\em A note on the exponential Diophantin equation $x^{2}-2^{m}=y^{n}$}, Proc. Amer. Math. Soc. 
{\bf 123} (1995), no.~12, 3627-3629.
%
\bibitem{LT1} F.~Luca, A. Togb\'{e} {\em On the Diophantine equation $x^{2}+2^{a}5^{b}=y^{n}$}, Int.~J.~Num.~Th. {\bf 4} (2008), no.~6, 973-979.
%
\bibitem{LMT} F.S.~Abu~Muriefah, F.~Luca, A. Togb\'{e} {\em On the Diophantine equation $x^{2}+5^{a}13^{b}=y^{n}$}, Glasgow Math.~J. {\bf 50} (2008), no.~1, 175-181.
%
\bibitem{LT2} F.~Luca, A. Togb\'{e} {\em On the Diophantine equation $x^{2}+2^{a}13^{b}=y^{n}$}, Colloq.~Math. {\bf 116} (2009), no.~1, 139-146.
%
\bibitem{P} I.~Pink, {\em On the Diophantine equation $x^{2}+2^{\alpha}3^{\beta}5^{\gamma}7\delta=y^{n}$}, 
Publ.~Math.~Deb. {\bf 70} (2007), no.~1-2, 149-166.
%
\bibitem{PR} I.~Pink, Z.~R\'{a}bai {\em On the Diophantine equation $x^{2}+5^{k}17^{l}=y^{n}$}, Comm.~in ~Math. {\bf 19} (2011), no.~1, 1-9.
%
\bibitem{Ri} K.~A.~Ribet, {\em On modular representations of $Gal(\overline{\Q}/\Q)$ arising from modular forms}, Invent. Mat. {\bf 100} (1990), no.~2, 431-476.
%
\bibitem{Si} S.~Siksek, {\em The modular approach to Diophantine 
equations}, Panoramas \& Synth\`{e}ses {\bf 36} (2012), 151-179.
%
\bibitem{S1} G.~Soydan, {\em On the Diophantine equation $x^{2}+7^{\alpha}11^{\beta}=y^{n}$}, Miskolc Math.~Notes {\bf 13} (2012), no.~2, 515-527.
%
\bibitem{S2} G.~Soydan, {\em Corrigendum to "On the Diophantine equation $x^{2}+7^{\alpha}11^{\beta}=y^{n}$"}, ibid. {\bf 15} (2014), no.~1, 217.
%
\bibitem{ST} G.~Soydan, N.~Tzanakis, {\em Complete solution of the Diophantine equation $x^{2}+5^{a}11^{b}=y^{n}$}, Bull. of the Hellenic Math. Soc. {\bf 60} (2016), 125-151.
%
\bibitem{SUZ} G.~Soydan, M.~Ulas, H.~Zhu {\em On the Diophantine equation $x^{2}+2^{a}19^{b}=y^{n}$}, Indian J.~Pure and App.~Math. {\bf 43} (2012), no.~3, 251-261.
%
\bibitem{Wiles} A.~Wiles, {\em Modular elliptic curves and Fermat's Last Theorem}, Ann.~of~Math.} {\bf 141} (1995), no.~3, 443-551.
%
\bibitem{ZLST} H.~Zhu, M.~Le, G.~Soydan, A. Togb\'{e} {\em On the exponential Diophantine equation 
$x^{2}+2^{a}p^{b}=y^{n}$}, Periodica Math.~Hung. {\bf 70} (2015), no.~2, 233-247.     
%
\end{thebibliography}
\end{document}